\documentclass{amsart}

\usepackage{graphicx}

\usepackage[arrow,matrix,graph,frame,poly,arc,tips]{xy}
\usepackage{amscd,amssymb,subfigure,hyperref,epsfig,color}
\usepackage[width=5.65in,height=8.0in,centering]{geometry}
\usepackage{amsmath}

\newtheorem{theorem}{Theorem}[section]
\newtheorem{lemma}[theorem]{Lemma}

\theoremstyle{definition}
\newtheorem{definition}[theorem]{Definition}
\newtheorem{example}[theorem]{Example}

\newtheorem{proposition}[theorem]{Proposition}
\newtheorem{corollary}[theorem]{Corollary}
\newtheorem{conjecture}[theorem]{Conjecture}

\theoremstyle{remark}
\newtheorem{remark}[theorem]{Remark}

\theoremstyle{problem}

\def\FF{\mathbb{F}}

\def\ZZ{\mathbb{Z}}

\def\cal{\mathcal}  
\newcommand{\A}{{\cal A}}

\def\abs#1{\left|#1\right|}                 

\numberwithin{equation}{section}

\begin{document}

\title[Skeleton Simplicial Evaluation Codes]{Skeleton Simplicial Evaluation Codes}


\author{James Berg}
\address{United States Naval Academy, Annapolis MD, USA}
\curraddr{}
\email{m110588@usna.edu}

\author{Max Wakefield}
\address{United States Naval Academy, Annapolis MD, USA}
\email{wakefiel@usna.edu}


\date{}

\dedicatory{}
\maketitle

\begin{abstract}
For a subspace arrangement over a finite field we study the evaluation code defined on the arrangements set of points. The length of this code is given by the subspace arrangements characteristic polynomial. For coordinate subspace arrangements the dimension is bounded below by the face vector of the corresponding simplicial complex. The minimum distance is determined for coordinate subspace arrangements where the simplicial complex is a skeleton. 
\end{abstract}

\section{Introduction} Evaluation codes have provided a solid foundation for interactions between commutative algebra, algebraic geometry, and coding theory. It is particularly fruitful to use tools in algebraic geometry to create efficient codes and enumerate their properties, for example see \cite{TV-91}, \cite{H-92}, \cite{GLS-05},  and \cite{DRT-01}. Evaluation codes associated to Order Domains and valuation rings have also shown to give remarkable results, for example see \cite{CLO-98}, \cite{GP-02}, and \cite{O-01}. On the other hand through the work of Relinde Jurrius and Ruud Pellikaan in \cite{RJRP1} and Stefan Toh{\v{a}}neanu in \cite{T-09} and \cite{T-10} the theory of hyperplane arrangements and subspace arrangements has shown be particularly useful in coding theory. 

However very little work has focused on defining Evaluation codes by hyperplane or subspace arrangements. This is the subject of the paper. In particular, the focus is on coordinate arrangements where the associated defining ideal is a square free monomial ideal called the Stanley-Riesner ideal. The main idea is to use basic results on subspace arrangements, Stanley-Riesner rings, and simplicial complexes to understand the associated evaluation codes. 

In this section we introduce the codes and state the main result, Theorem \ref{main}, which gives the minimum distance for degree one binary skeleton simplicial evaluation codes. Section \ref{proof} gives the proof of of Theorem \ref{main}. Then Section \ref{j>1} studies higher degree version of binary skeleton simplicial evaluation codes and lists a conjecture. Finally Section \ref{Hamming} shows how these codes are related to Hamming codes.

\subsection{Subspace Arrangements}
Let $V$ be a vector space of dimension $\ell$ over a finite field $\FF$ of $q$ elements. A subspace arrangement $\A=\{X_1,\dots ,X_t\}$ in $V$ is a finite collection of linear subspaces $X_i\subseteq V$. For a general reference for subspace arrangements see \cite{Bjo}.  Let $S=\FF_q [x_1,\dots ,x_\ell]$ be the symmetric algebra of the dual vector space $V^*$. Denote the points of $V$ in $\A$ by $P(\A)=\bigcup\limits_{i=1}^t X_i=\{p_1,\dots ,p_n\}$.  Additionally, let $I(\A)=\{ f\in S| f(P(\A))=0\}$ be the radical defining ideal of the variety $P(\A)$. 

Let $L(\mathcal{A})$ consist of all intersections of the subspaces of $\A$ (note that the empty intersection is defined as the entire vector space, $V$).  Then since every subspace contains the origin $L(\mathcal{A})$ is a lattice and a poset by reverse inclusion.  Next, the M\"obius function, $\mu$, on $L(\A)$ is $\mu :L(\mathcal{A})\longrightarrow\ZZ$ defined recursively by
$\mu(V)=1$ and $\mu(X)=-\sum\limits_{Y\lneq X} \mu(Y)$. From this function, the characteristic polynomial for $\A$, $\chi (\mathcal{A},t)$, is defined as $$\chi (\mathcal{A},t)=\sum_{X\in L(\mathcal{A})} \mu(X)t^{\mathrm{dim}(X)}.$$ In \cite{CAA}, Christos Athanasiadis proved that the characteristic polynomial determines the number of the set of points of $\mathcal{A}$: $$|P(\mathcal{A})|=q^\ell-\chi (\mathcal{A},q).$$

\subsection{Definition of $C(\A,j)$}\label{evalcode} 
Define the evaluation map $ev_\A :S_{\leq j}\to \FF^n$ by $$ev_\A (f)=(f(p_1),\dots,f(p_n))$$ where $S_{\leq j}$ is the vector space of all polynomials of less degree than or equal to $j$ in $S=\FF_q [x_1,\dots ,x_\ell]$. Now we can define our main object of study.
\begin{definition}
 The image $C(\A,j)=im(ev_\A)$ is a linear subspace in $\FF_q^n$ that we call a \emph{subspace arrangement code}.
\end{definition}

Now as a direct result of the Athanasiadis' counting theorem in \cite{CAA} we have the following corollary.

\begin{corollary}\label{length}If $n$ is the length of a code associated with the subspace arrangement $\A$, with characteristic polynomial $\chi (\mathcal{A},q)$, then $$n=|P(\mathcal{A})|=q^\ell-\chi (\mathcal{A},q).$$\end{corollary}

\begin{figure}
 \includegraphics[width=2in]{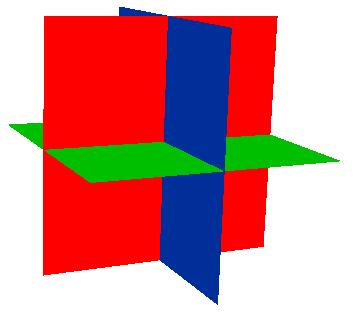}\caption{The coordinate planes} \label{coordinate}
\end{figure}
\begin{example}\label{[7,4,3]}
Let  $q=2$, $\ell=3$, $j=1$, and $\A$ be the $xy$-, $xz$-, and $yz$-planes, as seen in Figure \ref{coordinate} as viewed over the real numbers.  Then, $I(\A)={<x_1x_2x_3>}$.  Thus, $V=\FF_2^3$ and $S=\FF_q[x_1,x_2,x_3]$. The characteristic polynomial is $\chi (\A,t)=(t-1)^3$. The points of the arrangement are $P(\A)=\{(0,0,0),$ (1,0,0), (1,1,0), (1,0,1), (0,1,0), (0,1,1), $(0,0,1)\}$, so $|P(\A)|=7=2^3-\chi (\A,2)$.  Now, to find im$(ev_\A: S_{\leq 1} \rightarrow \FF_2^7)$, we write a basis for the subspace spanned by the image of the evaluation map as a matrix (the top row designates a point in $P(\A)$, and the first column delineates the polynomials in $S_{\leq 1}$ at which each point is evaluated):

$$\begin{array}{c|ccccccc}
  {} &(0,0,0)&(1,0,0)&(1,1,0)&(1,0,1)&(0,1,0)&(0,1,1)&(0,0,1)
\\
\hline
1  &  1    &   1   &   1   &   1   &   1   &   1   &   1
\\x_1&  0    &   1   &   1   &   1   &   0   &   0   &   0
\\x_2&  0    &   0   &   1   &   0   &   1   &   1   &   0
\\x_3&  0    &   0   &   0   &   1   &   0   &   1   &   1

\end{array}.$$

Close observation yields that the dimension is 4 and the minimum distance is 3.  Therefore, this code is a $[7,4,3]_2$ code and is hence permutation equivalent to the $[7,4,3]$ Hamming code.

\end{example}

\begin{remark}
 Note that if $\A$ is the entirety of $V=\FF_q^\ell$ (so $I(\A)={<0>}$) and $q=2$, then $P(\A)=V$.  Thus, the subspace arrangement code is equivalent to a Reed-Muller code.  Specifically, the code is $C(\FF_q^\ell,j)=\mathcal{R}(j,\ell)$.  Since skeletal codes (as discussed in Section \ref{skeleton}) are a punctured $C(\FF_q^\ell,j)$ code, they are also a punctured Reed-Muller code.
\end{remark}

\subsection{Simplicial Complexes}
We want to study the dimension and minimum distance of $C(\A,j)$ for any subspace arrangement $\A$, but doing so is difficult.  For the rest of the paper we will focus on subspace arrangements that are coordinate arrangements. We follow the standard formulation of the correspondence between coordinate arrangements and simplicial complexes written in \cite{Bjo} by Bj\"orner. For convenience, let $[k]=\{1,\ldots ,k\}$ be the set of numbers 1 through $k$. A \emph{simplicial complex}, $\Delta$, is a set of subsets of $[k]$ such that $$\begin{array}{c}(1) \,\mathrm{if}\, \sigma\in\Delta \,\mathrm{and}\,\tau\mathrm{\ is\ a\ subset\ of\ }\sigma,\,\mathrm{then}\, \tau\in\Delta,\,\mathrm{and} \\ (2)\, \mathrm{if}\, x\in [k],\, \mathrm{then}\, \{x\}\in\Delta\end{array}$$ where $[k]$ are called the vertices of $\Delta$ and the $\sigma$ are called the faces of $\Delta$. A $k$-simplex is a simplicial complex that contains all subsets of $[k]$. 

For a simplicial complex $\Delta$ with vertices $[\ell]$ and for any face $\sigma=\{i_1,\dots,i_s\} \in \Delta$ let $x^\sigma =x_{i_1}x_{i_2}\cdots x_{i_s}$. The Stanley-Riesner ideal of $\Delta$ is the ideal $I_\Delta=\{x_\sigma$ $|$ $\sigma\not\in \Delta\}\subseteq S$. For $\{\mathbf{b}_1,\dots,\mathbf{b}_\ell\}$ (a basis of $\FF_q^\ell$) and each subset $\sigma=\{i_1,\dots,i_s\}\subseteq[\ell]$, let $X_\sigma=\mathrm{span}\{\mathbf{b}_{i_1},\dots,\mathbf{b}_{i_s}\}$.  Then the coordinate subspace arrangement corresponding to $\Delta$ is $\mathcal{A}_\Delta=\{X_\sigma|\sigma\in \Delta\}$. 

\begin{definition}
 Let $\Delta$ be a simplicial complex with corresponding subspace arrangement $\A_\Delta$.  Then the \emph{simplicial evaluation code} corresponding to $\Delta$ is the subspace arrangement code $C(\A_\Delta,j)$ defined using the Stanley-Riesner ideal $I_\Delta$.
\end{definition}

\begin{remark}
The arrangement in Example \ref{[7,4,3]} is the coordinate arrangement who's simplicial complex is the empty triangle. 
\end{remark}

The dimension of $C(\A,j)$ for an arbitrary subspace arrangement or even $C(\A_\Delta ,j)$ can be very difficult to find. Hence for the remainder of this note we focus on the case when the size of the field is $q=2$. 

\subsection{Dimension}  In order to compute the dimension of these codes we will need the notion of a face vector. A \emph{face vector}, $[f_i]$, of $|Delta$ is the number of faces of dimension $i-1$ in $\Delta$.  That is, $f_i=|\{ \sigma \in \Delta : |\sigma|= i\}|$ for $0< i$ where we define $f_0:=1$.

\begin{proposition}\label{dimbsc} For binary simplicial evaluation codes the dimension is $k=\sum\limits^s_{i=0} f_i$ where $s=\mathrm{min}(j,\mathrm{dim}(\Delta))$.\end{proposition}
\proof Any $i$-face of $\Delta$ corresponds to a basis element of the subspace $C(\A_\Delta ,j)$ as long as $0\leq i\leq j$. And any $i$-face of $\Delta$ corresponds to a basis element of the vector space $S/I_{\Delta}$. Then one can construct an upper triangular generator matrix of $C(\A_\Delta ,j)$ by listing the rows by the basis elements of $S/I_{\Delta}$ in degree lexicographic ordering and by listing the columns corresponding to points similarly. Then it is well known that the Hilbert function of the algebra $S/I_{\Delta}$ is given by the face vector.\qed

\begin{corollary}
 Let $M$ denote the dimension of the minimum dimensional non-face in a simplicial complex $\Delta$.  Then, if $j\leq M-1$, the dimension of $C(\A_\Delta ,j)$ is $k=\sum\limits_{m=0}^j {\ell \choose m}$.
\end{corollary}

While a convenient formula for dimension has been easily determined, the case for minimum distance is much more difficult. Hence for the remainder of this note we reduce our attention to only the following type of simplicial complex.

\begin{definition}\label{1}For $0\leq h\leq \ell$ an \emph{h-skeleton} is a simplicial complex, denoted by $\Delta (\ell,h)$, on $\ell$ vertices consisting of all possible $h-1$ to 0-dimensional faces.\end{definition}

\begin{definition}A binary \emph{h-skeleton code} is the binary evaluation code $C(\A_{\Delta(\ell ,h)},j)_2$ of the associated coordinate arrangement to the $h$-skeleton and is denoted $K(\ell , h,j).$\end{definition}

The length of these codes is obtained by counting the number of points that have at most $h$ non-zero entries. Hence, the length of $K(\ell ,h,j)$ is \begin{equation}\label{k-length}n=\sum\limits_{i=0}^h{ \ell \choose i}.\end{equation} The dimension can be found as a specialization of Proposition \ref{dimbsc}.

\begin{corollary}\label{k-dimension}

For $0\leq j\leq h\leq \ell$ $$\dim (K(\ell ,h,j))=\sum\limits_{i=0}^j{\ell \choose i}.$$

\end{corollary}

The main result of this paper is the calculation of the minimum distance of these codes for $j=1$.

\begin{theorem}\label{main} The minimum distance of the codes $C(\A_{\Delta(\ell,h)},1)=K(\ell,h,1)$ is $$\sum\limits_{a=1}^{h}\binom{\ell -1}{a-1}.$$\end{theorem}

The proof of Theorem \ref{main} occupies the entirety of Section \ref{proof}. It is elementary using basic counting techniques. A shorter proof using generating functions might be possible. Section \ref{j>1} focuses on the minimum distance of the codes $K(\ell ,h,j)$ for $j>1$. Some bounds are given and a conjecture is given, but the minimum distance for $j>1$ is at this time out of hand for the authors. Section \ref{Hamming} investigates the relationship between $K(\ell,h,j)$ and Hamming codes. There it is shown that for certain values of $h$ and $j$ these codes are permutation equivalent.

\section{Proof of Theorem \ref{main}}\label{proof}

In order to obtain more information about $K(\ell,h,j)$ we need to examine and carefully construct a convenient generating matrix. To do this we need a little notation. Let $\sigma=\{i_1,\dots ,i_r\} \subseteq [\ell ]$ and let $x_\sigma =x_{i_1}x_{i_2}\cdots x_{i_r}$. With this notation, the Stanley-Reisner ideal of $\Delta (\ell ,h)$ is $$I_{\Delta (\ell ,h)}=(x_\sigma : |\sigma |=h+1).$$ To denote points in the arrangement $\A_{\Delta (\ell,h)}$ we let $\{e_1\dots ,e_\ell\}$ be the standard basis for $V=\FF_2^\ell$ (that is, $e_i$ has all components 0 except a 1 in the $i$-th coordinate). Now for $\tau=\{i_1,\dots ,i_s\} \subseteq [\ell]$ let $e_\tau=\sum\limits_{k=1}^se_{i_k}$. Then the set of all points in the arrangement $\A_{\Delta(\ell ,h)}$ is $$\left[ \bigcup\limits_{X\in \A_{\Delta(\ell,h)}}X\right]=\{e_\tau : 1\leq |\tau | \leq h\}.$$

Now we will construct the blocks of the generating matrix for $K(\ell ,h,j)$. Let $B_{rs}$ be the matrix defined as \begin{equation}B_{rs}=(x_\sigma (e_\tau))\end{equation} where $|\sigma|=r$, $|\tau |=s$, and the rows and columns are ordered degree lexicographically. Then a generating matrix $G(\ell ,h,j)$ of $K(\ell,h,j)$ constructed block-wise is $$G(\ell ,h,j)=(B_{rs})_{\substack{0\leq r\leq j\\ 0\leq s\leq h}}.$$ We can now denote column and row blocks of the generating matrix.

\begin{definition} Let $CB_t=\{B_{rt}: 0\leq r\leq h\}$ be the union of the blocks of columns in the matrix of the code with $t$ ones in each point. Let $RB_t=\{B_{tr}: 0\leq r\leq j\}$ be the union of the blocks of rows in the matrix of the code with $t$ ones in each point.\end{definition}

This notation for the generating matrix streamlines the computation of minimum distance. We begin by presenting an upper bound for the minimum distance.

\begin{lemma}\label{1-row} For $0\leq j\leq h\leq \ell$ the minimum distance of $K(\ell,h,j)$ satisfies $$d\leq \sum\limits_{i=0}^{h-j}{\ell -j \choose i}.
$$\end{lemma}

\begin{proof}

In the generating matrix $G(\ell, h,j)$ the rows in the last row block $RB_j$ have the smallest weight. The smallest $t$ such that $B_{jt}$ has no zero entries is when $t=j$. The Hamming weight of any row of $B_{jt}$ for $j\leq t\leq h$ is ${\ell -j \choose t-j}$. Hence, the Hamming weight of an entire row in $RB_j$ is $$\sum\limits_{i=j}^h{\ell -j \choose i-j}.$$ \end{proof}

\begin{remark} If $j=h$ then $K(\ell,h,j)$ is a maximum distance separable code but the minimum distance is 1 because the upper bound here is 1. \end{remark}

If $j=1$ then the minimum distance is bounded by $d\leq \sum\limits_{i=0}^{h-1}\binom{\ell-1}{i}$. The main result of this paper (Theorem \ref{main}) is that this upper bound is exactly the minimum distance for the case $j=1$. First, we obtain a formula for the Hamming weight of adding $s$ rows of the generating matrix. In order to develop this formula, we need a little more notation. Suppose $x_{i_1}, \dots , x_{i_s}$ are the degree one monomials that correspond to the $s$ rows we are to sum in $B_{1a}$. Let $P_a$ be the set of all points in $\FF_2^\ell$ that have exactly $a$ nonzero entries. Note that $P_a$ corresponds to the columns of $B_{1a}$. 

\begin{definition}\label{XL}

For $1\leq t\leq s$, let $X_r:=\{ p\in P_a: x_{i_r}(p)=1\}$. Let $\mathcal{L}^{a,s}_t$ be the set of all the sets of points that evaluate to 1 on at least $t$ degree one monomials, so $$\mathcal{L}^{a,s}_t=\{X_{k_1}\cap \cdots \cap X_{k_t}: \{k_1,\dots , k_t\}\subseteq \{i_1,\dots ,i_s\}\}.$$

\end{definition}

If we wanted to calculate the size of the union of the sets $X_{i_1}\cup \cdots \cup X_{i_s}$, then we could use a standard inclusion-exclusion formula $$|X_{i_1}\cup \cdots \cup X_{i_s}|=\sum\limits_{t=1}^s(-1)^{t+1}\sum\limits_{Y\in \mathcal{L}^{a,s}_t}|Y|.$$ However, we want to calculate the Hamming weight of the sum of these row vectors of which not all points will sum to 1. To do this we will create a generalized inclusion-exclusion formula. 

\begin{lemma}\label{-2}

The Hamming weight of adding $s$ row vectors of $B_{1a}$ of the code $C(\A_{\Delta(l,h)},1)$ is $$\sum\limits_{t=1}^s(-2)^{t-1}\sum\limits_{Y\in \mathcal{L}^{a,s}_t}|Y|.$$

\end{lemma}

\begin{proof}

We prove this by induction. The critical idea here is that if a point $p$ is contained in exactly $t$ sets $X_{k_1},\dots, X_{k_t}$ and not in any others, then the entry corresponding to this point in the sum will be 0 if $t$ is even and 1 if $t$ is odd. Let $c_t$ be the coefficient that will be multiplied to the point $p$ that is contained in exactly $t$ sets $X_{k_1},\dots, X_{k_t}$ in the sum (note that points are the objects being summed here because the $Y$s consist of points).  In the case when $t=1$, we want to count all the points that are in exactly 1 set. We therefore sum the entirety of the sets of just one intersection: $\sum\limits_{r=1}^s|X_{i_r}|$. Thus, the coefficient is $c_1=1$ for the $t=1$ term. However, if $t>1$, the point $p$ has already been counted in lower terms because it is also a subset of  all possible intersections of these $t$ sets: $$X_{k_1},\dots ,X_{k_t},X_{k_1}\cap X_{k_2},\dots ,X_{k_{t-1}}\cap X_{k_t},\dots ,X_{k_1}\cap \cdots \cap X_{k_{t-1}}, \dots ,X_{k_2}\cap \cdots X_{k_t}.$$ Because we want the coefficient for $t$ odd to be 1 and for $t$ even to be zero, we now have that  $$\sum\limits_{r=1}^t{t\choose r}c_r=\left\{ \begin{array}{ccc} 1 & \ \ & t \text{ odd}\\ 0 & \ \ & t \text{ even } \\ \end{array} \right. .$$ One method to do this is to set $$\sum\limits_{r=1}^t{t\choose r}c_r=\frac{(-1)^t-1}{-2}.$$ Now we prove by induction on $t$ that $c_t=(-2)^{t-1}$. The base is already provided above. 

By construction, $c_{t+1}=\frac{(-1)^{t+1}-1}{-2}-\sum\limits_{r=1}^{t}\binom{t+1}{r} c_r$. Then by the induction hypothesis, $$c_{t+1}=\frac{(-1)^{t+1}-1}{-2}-\sum\limits_{r=1}^{t}\binom{t+1}{r} (-2)^{r-1}.$$  Using the binomial expansion formula, we see that that $$(-1)^{t+1}=(-2+1)^{t+1}=\sum\limits_{i=0}^{t+1} \binom{t+1}{i} (-2)^i 1^{t+1-i}=\sum\limits_{i=0}^{t+1} \binom{t+1}{i} (-2)^i $$ $$= 1+\sum\limits_{i=1}^{t+1} \binom{t+1}{i} (-2)^i=1+(-2)\sum\limits_{i=1}^{t+1} \binom{t+1}{i} (-2)^{i-1}.$$ Hence, $$\frac{(-1)^{t+1}-1}{-2}=\sum\limits_{i=1}^{t+1} \binom{t+1}{i} (-2)^{i-1} = (-2)^{t} + \sum\limits_{i=1}^{t} \binom{t+1}{i} (-2)^{i-1}.$$ Now add the sum to both sides of this equation to obtain  $$(-2)^{t} = \frac{(-1)^{t+1}-1}{-2}-\sum\limits_{i=1}^{t} \binom{t+1}{i} (-2)^{i-1}=c_t.$$ \end{proof}

Lemma \ref{-2} gives a nice method to compute the Hamming weight of the sum of $s$ row vectors.

\begin{lemma}\label{s-rows}The Hamming weight of adding $s$ vectors in $RB_1$ is $$\sum\limits_{a=1}^{h}\sum\limits_{t=1}^s(-2)^{t-1}\binom{s}{t}\binom{\ell-t}{a-t}.$$\end{lemma}
\begin{proof} Note that for any $Y\in \mathcal{L}^{a,s}_t$, the size is $|Y|={ \ell -t \choose a-t}$ since $t$ of the nonzero entries must match up with the $t$ monomials (so there are $a-t$ ones left to chose from the remaining $\ell-t$ components of the point). Since we can choose any $t$ subsets of the monomials, we have $$|\mathcal{L}^{a,s}_t|={ s \choose t}.$$ Then the formula for the Hamming weight is given by applying Lemma \ref{-2} and summing over all possible column blocks $CB_a$ where $1\leq a\leq h$.\end{proof}

Next we prove a technical lemma that will be used in the proof of the main theorem.

\begin{lemma}\label{tech}

If $g_i^s=\sum\limits_{t=1}^i{s-t \choose i-t}{s \choose t}(-2)^{t-1}$, then $$g_i^s=\left\{ \begin{array}{cc} \binom{s}{i} & 2|i \\ 0 & 2 \nmid i\end{array}\right..$$

\end{lemma}

\begin{proof}

We prove this in two cases. Case 1 is when $i=2m$. Then $$g_{2m}^s=\sum\limits_{t=1}^{2m}\binom{s-t}{2m-t}\binom{s}{t}(-2)^{t-1}=\sum\limits_{t=1}^{2m} \frac{(s-t)!s!}{(2m-t)!(s-2m)!t!(s-t)!} (-2)^{t-1}$$$$=\sum\limits_{t=1}^{2m} \frac{s!}{(2m-t)!(s-2m)!t!}\frac{(2m)!}{(2m)!} (-2)^{t-1}=\frac{s!}{(s-2m)!(2m)!}\sum\limits_{t=1}^{2m} \frac{(2m)!}{(2m-t)!t!} \left(-2\right)^{t-1}$$$$=\binom{s}{2m}\left(-\frac{1}{2}\right)\left[\sum\limits_{t=1}^{2m}\binom{2m}{t}(-2)^t\right]=\binom{s}{2m}\left(-\frac{1}{2}\right)\left[\sum\limits_{t=0}^{2m}\binom{2m}{t}(-2)^t-1\right]$$$$=\binom{s}{2m}\left(-\frac{1}{2}\right)\left[\sum\limits_{t=0}^{2m}\binom{2m}{t}(-2)^t(1)^{2m-t}-1\right]=\binom{s}{2m}\left(-\frac{1}{2}\right)\left[ (1-2)^{2m}-1\right]$$$$=\binom{s}{2m}\left(-\frac{1}{2}\right) (1-1)=0.$$

Case 2 is when $i=2m+1$. Then $$g_{2m+1}^s=\sum\limits_{t=1}^{2m+1}\binom{s-t}{2m+1-t}\binom{s}{t}(-2)^{t-1}$$$$=\sum\limits_{t=1}^{2m+1} \frac{(s-t)!s!}{(2m+1-t)!(s-2m-1)!t!(s-t)!} (-2)^{t-1}$$$$=\sum\limits_{t=1}^{2m+1} \frac{s!}{(2m+1-t)!(s-2m-1)!t!}\frac{(2m+1)!}{(2m+1)!} (-2)^{t-1}$$$$=\frac{s!}{(s-2m-1)!(2m+1)!}\sum\limits_{t=1}^{2m+1} \frac{(2m+1)!}{(2m+1-t)!t!} (-2)^{t-1}$$$$=\binom{s}{2m+1}\left(-\frac{1}{2}\right)\left[\sum\limits_{t=1}^{2m+1}\binom{2m+1}{t}(-2)^t\right]$$$$=\binom{s}{2m+1}\left(-\frac{1}{2}\right)\left[\sum\limits_{t=0}^{2m+1}\binom{2m+1}{t}(-2)^t-1\right]$$$$=\binom{s}{2m+1}\left(-\frac{1}{2}\right)\left[\sum\limits_{t=0}^{2m+1}\binom{2m+1}{t}(-2)^t(1)^{2m+1-t}-1\right]$$$$=\binom{s}{2m+1}\left(-\frac{1}{2}\right)\left[(1-2)^{2m+1}-1\right]=\binom{s}{2m}\left(-\frac{1}{2}\right) (-1-1)$$$$=\binom{s}{2m}\left(-\frac{1}{2}\right) (-2)=\binom{s}{2m}.$$\end{proof}

Notice that the formula in Lemma \ref{s-rows} for the case $s=1$ is exactly the computation made in Lemma \ref{1-row}. In order to show that this value for $s=1$ is exactly the minimum distance, it is enough to show that the sum for $s>1$ is greater than that for $s=1$, since $q=2$. 

\begin{proof}[Proof of Theorem \ref{main}]  We need to show that the Hamming weight of adding $s$ rows given in Lemma \ref{s-rows} is always larger than the Hamming weight of one row given in Lemma \ref{1-row}: $$\sum\limits_{a=1}^{h}\sum\limits_{t=1}^s(-2)^{t-1}\binom{s}{t}\binom{\ell-t}{a-t} \geq \sum\limits_{a=1}^{h}{\ell -1 \choose a-1}$$ This is equivalent to showing \begin{equation}\label{dude}\sum\limits_{a=1}^{h}\left[ \left(\sum\limits_{t=1}^s(-2)^{t-1}\binom{s}{t}\binom{\ell-t}{a-t}\right) -{\ell -1 \choose a-1}\right]\geq 0.\end{equation} Now we use Pascal's formula to allow for the exchange of terms of \ref{dude}. We examine the term $${\ell -t \choose a-t}=\binom{\ell-t-1}{a-t}+\binom{\ell-t-1}{a-t-1}$$ $$=\left(\binom{\ell-t-2}{a-t}+\binom{\ell-t-2}{a-t-1}\right)+\left( \binom{\ell-t-2}{a-t-1}+\binom{\ell-t-3}{a-t-2}\right) =\dots$$  Since, whenever Pascal's formula is used, each binomial coefficient is broken down into two binomial coefficients, the process is analogous to Pascal's triangle: the top number, $l-t-x$, corresponds to the $x$th row, and the bottom number, $a-t-x$, corresponds to the $x$th column.  Thus, there are $\binom{s-t}{i-t}$ occurrences of each $\binom{\ell-s}{a-i}$ for each $t$.  Therefore, $\binom{\ell-t}{a-t}=\sum\limits_{i=1}^s \binom{\ell-s}{a-i}\binom{s-t}{i-t}$, so the inequality we are trying to prove is now \begin{equation}\label{expanded}\sum\limits_{a=1}^{h} \left[\left(\sum\limits_{t=1}^s (-2)^{t-1}\binom{s}{t}\sum\limits_{i=1}^s \binom{\ell-s}{a-i}\binom{s-t}{i-t}\right)-\sum\limits_{i=1}^s \binom{\ell-s}{a-i}\binom{s-1}{i-1}\right] \geq 0.\end{equation} Now focusing on the ${\ell -s \choose a-s}$ terms, \ref{expanded} becomes \begin{equation}\label{terms} \sum\limits_{a=1}^h \left[ \sum\limits_{i=1}^s \left(\left( \sum\limits_{t=1}^s (-2)^{t-1}\binom{s}{t}{s-t \choose i-t}\right) -{s-1\choose i-1}\right) {\ell -s\choose a-i}\right] \geq 0\end{equation} Notice that the third sum is only nonzero when $t\leq i$ and that the ${s-1\choose i-1}$ term only affects the $t=1$ term of the third sum. Hence, we can rewrite \ref{terms} as \begin{equation}\label{factored}  \sum\limits_{a=1}^h \left[ \sum\limits_{i=1}^s \left( (s-1){s-1\choose i-1}+\sum\limits_{t=2}^i (-2)^{t-1}\binom{s}{t}{s-t \choose i-t}\right) {\ell -s\choose a-i}\right] \geq 0\end{equation} Let $d^s_i $ be the coefficient of ${\ell -s\choose a-i}$ in \ref{factored}: $$d_i^s=(s-1){s-1\choose i-1}+\sum\limits_{t=2}^i (-2)^{t-1}\binom{s}{t}{s-t \choose i-t}.$$ Recall the numbers $g_i^s$ from Lemma \ref{tech}: $$g_i^s=\sum\limits_{t=1}^i{s-t \choose i-t}{s \choose t}(-2)^{t-1}.$$ Then $$g_i^s-d_i^s=\sum\limits_{t=1}^i{s-t \choose i-t}{s \choose t}(-2)^{t-1}-(s-1){s-1\choose i-1}-\sum\limits_{t=2}^i (-2)^{t-1}\binom{s}{t}{s-t \choose i-t}$$ $$={s-1\choose i-1}.$$ Thus, \begin{equation}\label{switch} d_{2m+1}^s=g_{2m+1}^s-{s-1\choose 2m},\end{equation} which, by Lemma \ref{tech}, gives that \ref{switch} becomes $$d_{2m+1}^s=-{s-1\choose 2m}.$$ Using Lemma \ref{switch} again, we get $$d_{2m}^s=g_{2m}^s-{s-1\choose 2m-1}={s \choose 2m}-{s-1 \choose 2m-1}={s-1\choose 2m}.$$ Hence, \begin{equation}\label{tele}d_{2m+1}^s=-d_{2m}^s.\end{equation}

The main inequality we are trying to prove, \ref{factored}, can now be written as \begin{equation}\label{d's} \sum\limits_{a=1}^h\left[ \sum\limits_{i=1}^sd_i^s{\ell -s \choose a-i}\right] \geq 0.\end{equation} Assume $s=2m$ is even and expand the left hand side of \ref{d's} via odds and evens: $$\sum\limits_{a=1}^h\left[ \left( \sum\limits_{r=1}^{m}d_{2r}^{2m}{\ell -2m \choose a-2r}\right) +\left(\sum\limits_{r=0}^{m-1}d_{2r+1}^{2m}{\ell -2m\choose a-2r-1}\right)\right] .$$ Then using \ref{d's} on the odd terms, we get \begin{equation}\label{tele2} \sum\limits_{a=1}^h\left[ \left( \sum\limits_{r=1}^{m}d_{2r}^{2m}{\ell -2m \choose a-2r}\right) -\left(\sum\limits_{r=0}^{m-1}d_{2r}^{2m}{\ell -2m\choose a-2r-1}\right)\right] \end{equation} Then using Pascal's formula on \ref{tele2}, we have \begin{equation}\label{tele3} \sum\limits_{a=1}^h\left[ d_{2m}^{2m}{\ell -2m\choose a-2m}-d_0^{2m}{\ell -2m\choose a-1}+\sum\limits_{r=1}^{m-1}d_{2r}^{2m}\left( {\ell -2m \choose a-2r}-{\ell -2m\choose a-2r-1}\right)\right] .\end{equation} Then switch sums on \ref{tele3} to get \begin{equation}\label{tele4} d_{2m}^{2m}\sum\limits_{a=1}^h{\ell -2m\choose a-2m}-d_0^{2m}\sum\limits_{a=1}^h{\ell -2m\choose a-1}+\left[\sum\limits_{r=1}^{m-1}d_{2r}^{2m}\sum\limits_{a=1}^h\left( {\ell -2m \choose a-2r}-{\ell -2m\choose a-2r-1}\right)\right] .\end{equation} Then the sum in the left portion of \ref{tele4} telescopes: \begin{equation}\label{telescope} d_{2m}^{2m}\sum\limits_{a=1}^h{\ell -2m\choose a-2m}-d_0^{2m}\sum\limits_{a=1}^h{\ell -2m\choose a-1}+\left[\sum\limits_{r=1}^{m-1}d_{2r}^{2m}\left(-{\ell -2m \choose -2r}+{\ell -2m\choose h-2r}\right)\right] .\end{equation} Then notice that $d_0^{2m}=0$ and that \ref{telescope} becomes \begin{equation}\label{BERG} d_{2m}^{2m}\sum\limits_{a=1}^h{\ell -2m\choose a-2m}+\left[\sum\limits_{r=1}^{m-1}d_{2r}^{2m}{\ell -2m\choose h-2r}\right].\end{equation}
Since \ref{BERG} is the left hand side of \ref{factored} and each term is positive, we have proved the theorem.\end{proof}

\begin{example}
 $\Delta(5,2)$consists of 5 vertices and all 1-dimensional and 0-dimensional faces (which are the edges and vertices, respectively).  This simplicial complex is graphically represented in Figure \ref{15simplex}.  The matrix generating $K(5,2,1)$ is as follows:
$$\left[\begin{array}{cccccccccccccccc}
1&1 &1 &  1&1&1&1&1&1&1&1&1&1&1&1&1\\
0&1 &0 &  0&0&0&1&1&1&1&0&0&0&0&0&0\\
0&0 &1 &  0&0&0&1&0&0&0&1&1&1&0&0&0\\
0&0 &0 &  1&0&0&0&1&0&0&1&0&0&1&1&0\\
0&0 &0 &  0&1&0&0&0&1&0&0&1&0&1&0&1\\
0&0 &0 &  0&0&1&0&0&0&1&0&0&1&0&1&1\\ 
\end{array}\right].$$  Thus, by the formulas given in  Equation \ref{k-length}, Corollary \ref{k-dimension}, and Theorem \ref{main}, $K(5,2,1)$ is a $[\sum\limits_{i=0}^h{\ell \choose i},\sum\limits_{i=0}^j{\ell \choose i}, \sum\limits_{a=1}^h \binom{\ell-1}{a-1}]_2=[\sum\limits_{i=0}^2{5 \choose i},\sum\limits_{i=0}^1{5 \choose i}, \sum\limits_{a=1}^2 \binom{5-1}{a-1}]_2=[16,6,5]_2$ code.
\end{example}

\begin{figure}
 \includegraphics[height=2in]{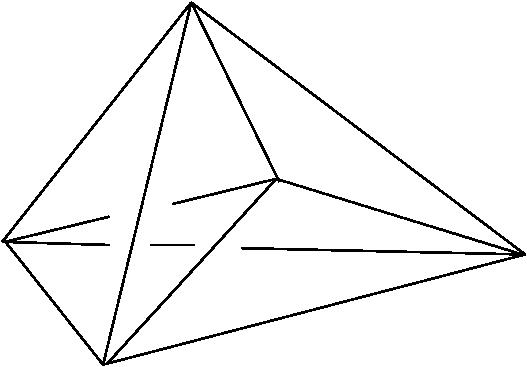}\caption{$\Delta(5,2)$}\label{15simplex}
\end{figure}

\section{The minimum distance of $K(\ell ,h,j)$ for $j>1$}\label{j>1}
Now we focus on the case when $j>1$. It is more complicated, and we are not able to calculate the minimum distance.  However, we are able to find formulas for summing row vectors of the generating matrix and and are able to compare these formulas to the conjectured upper bound. 

\begin{definition} Let $B(\sigma_1,\dots,\sigma_s)$ be the Hamming weight of adding $s$ rows of the generating matrix $G(\ell ,h,j)$ where each row corresponds to a set $\sigma_i\subseteq [\ell]=\{1,\dots,\ell \}$.\end{definition} The $j>1$ analogue to Definition \ref{XL} is the following.

\begin{definition}
 
For $1\leq t\leq s$ let $X_{\sigma_r}:=\{ p\in P_a: x_{\sigma_r}(p)=1\}$. Let $\mathcal{L}^{a,s}_t$ be the set of all the sets of points that evaluate to 1 on at least $t$ of the $s$ monomials $x_{\sigma_1},\dots,x_{\sigma_s}$. Thus, $$\mathcal{L}^{a,s}_t=\{X_{k_1}\cap \cdots \cap X_{k_t}: \{k_1,\dots , k_t\}\subseteq \{\sigma_1,\dots ,\sigma_s\}\}.$$\end{definition}

\begin{proposition}\label{size} For any skeletal code $K(l,h,j)$, the Hamming weight of adding $n$ rows of the generating matrix is $$B(\sigma_1,\dots,\sigma_n)= \sum\limits_{e=1}^n \left[(-2)^{e-1}\sum\limits_{\substack{I\subseteq[n]\\ \abs{I}=e}}B\left(\bigcup\limits_{i\in I}\sigma_i\right)\right] =\sum\limits_{e=1}^n \left[ (-2)^{e-1}\sum\limits_{\substack{I\subseteq[n]\\ \abs{I}=e}}\sum\limits_{i=0}^{h-\abs{\cup \sigma_i}}\binom{\ell-\abs{\cup \sigma_i}}{i}\right].$$\end{proposition}
\begin{proof} The $(-2)^{t-1}$ coefficient follows with the same argument as in Lemma \ref{-2}. Then we realize that the number of points in a $t$-fold intersection $Y=X_{k_1}\cap \cdots \cap X_{k_t}\in \mathcal{L}_t^{a,s}$ is equal to the Hamming weight of the row corresponding to the union of the sets $$\bigcup_{i=1}^tX_{k_i}.$$ Note that this row might not actually exist in the generating matrix. However, we can consider it as a row in the full matrix where $j=\ell$. Finally, the remainder of the formula is realized by applying Lemma \ref{1-row}. \end{proof}

In order to show that the minimum distance is equal to the upper bound, it must be demonstrated that \begin{equation}\label{hard}\sum\limits_{i=1}^{h-j}\binom{\ell-j}{i}\leq\sum\limits_{e=1}^n\left[(-2)^{e-1}\sum\limits_{\substack{I\subseteq[n]\\ \abs{I}=e}}\sum\limits_{i=0}^{h-\abs{\cup \sigma_i}}\binom{\ell-\abs{\cup \sigma_i}}{i}\right].\end{equation} However, this proof turns out to be difficult. For the remainder, we examine a few cases.

\begin{proposition}\label{base} For $j>1$, $B(\sigma_1,\sigma_2)\geq\sum\limits_{i=1}^{h-j}\binom{\ell-j}{i}$.\end{proposition}

\begin{proof}Let $\abs{\sigma_1}=i_1$, $\abs{\sigma_2}=i_2$.  Without the loss of generality, assume (relabeling as necessary) $i_1\leq i_2$ and $\sigma_1\neq \sigma_2$ (else, $B(\sigma_1,\sigma_2)=0$).  Note that $\abs{\sigma_1\cup \sigma_2}\geq i_1+1$.  By Proposition \ref{size}, $$B(\sigma_1,\sigma_2)=B(\sigma_1)+B(\sigma_2)-2B(\sigma_1\cup \sigma_2)$$ $$=\sum\limits_{i=1}^{h-i_1}\binom{\ell-i_1}{i}+\sum\limits_{i=1}^{h-i_2}\binom{\ell-i_2}{i}-2\sum\limits_{i=1}^{h-\abs{\sigma_1\cup \sigma_2}}\binom{\ell-\abs{\sigma_1\cup \sigma_2}}{i}.$$  Since $\abs{\sigma_1\cup \sigma_2}\geq i_1+1$, $$\sum\limits_{i=1}^{h-i_1}\binom{\ell-i_1}{i}+\sum\limits_{i=1}^{h-i_2}\binom{\ell-i_2}{i}-2\sum\limits_{i=1}^{h-\abs{\sigma_1\cup \sigma_2}}\binom{\ell-\abs{\sigma_1\cup \sigma_2}}{i}$$$$\geq \sum\limits_{i=1}^{h-i_1}\binom{\ell-i_1}{i}+\sum\limits_{i=1}^{h-i_2}\binom{\ell-i_2}{i}-2\sum\limits_{i=1}^{h-(i_1+1)}\binom{\ell-(i_1+1)}{i}.$$  Hence, we prove the inequality for $\abs{\sigma_1\cup \sigma_2}=i_1 +1$.  There are now two cases: (1) $i_1=i_2$ or (2) $i_1 +1=i_2$ and $\sigma_1\subseteq \sigma_2$.  For Case 1, assume $i_1=i_2=a$, so the left-hand side of the proposed inequality becomes $$2\left(\sum\limits_{i=1}^{h-a}\binom{\ell-a}{i}\right)-2\left(\sum\limits_{i=1}^{h-(a+1)}\binom{\ell-(a+1)}{i}\right).$$  Since $j\geq a$, $\sum\limits_{i=1}^{h-j}\binom{\ell-j}{i}\leq\sum\limits_{i=1}^{h-a}\binom{\ell-a}{i}$, so it suffices to show $$2\left(\sum\limits_{i=1}^{h-a}\binom{\ell-a}{i}\right)-2\left(\sum\limits_{i=1}^{h-(a+1)}\binom{\ell-(a+1)}{i}\right)\geq\sum\limits_{i=1}^{h-a}\binom{\ell-a}{i}.$$  This last expression can be rewritten as $$\sum\limits_{i=1}^{h-a}\binom{\ell-a}{i}-2\left(\sum\limits_{i=1}^{h-(a+1)}\binom{\ell-(a+1)}{i}\right)\geq0.$$  We can phrase this last inequality by saying, ``The Hamming Weight decreases by more than half in each group of rows (each group of rows corresponds to polynomials of the same degree)."  Since $\binom{n}{k}=\binom{n-1}{k}+\binom{n-1}{k-1}$, $$\sum\limits_{i=1}^{h-a}\binom{\ell-a}{i}-2\left(\sum\limits_{i=1}^{h-(a+1)}\binom{\ell-(a+1)}{i}\right)$$$$=\sum\limits_{i=1}^{h-a}\left(\binom{\ell-a-1}{i}+\binom{\ell-a-1}{i-1}\right)-2\left(\sum\limits_{i=1}^{h-a-1}\binom{\ell-a-1}{i}\right)$$$$=2\left(\sum\limits_{i=1}^{h-a-1}\binom{\ell-a-1}{i}\right)-2\left(\sum\limits_{i=1}^{h-a-1}\binom{\ell-a-1}{i}\right)+\binom{\ell-a-1}{0}+\binom{\ell-a-1}{h-a}$$$$=1+\binom{\ell-a-1}{h-a}\geq 0.$$  Thus, for Case 1, $B(\sigma_1,\sigma_2)\geq\sum\limits_{i=1}^{h-j}\binom{\ell-j}{i}$.  

For Case 2, since the Hamming Weight decreases by more than half in each group of rows, $$B(\sigma_1,\sigma_2)\geq B(\sigma_1)-B(\sigma_2)\geq\frac{1}{2}B(\sigma_1)\geq B(\sigma_2)\geq B(\sigma_j),$$ where $B(\sigma_j)=\sum\limits_{i=1}^{h-j}\binom{\ell-j}{i}$, since $\abs{\sigma_j}\geq\abs{\sigma_2}$.  Thus, $B(\sigma_1,\sigma_2)\geq\sum\limits_{i=1}^{h-j}\binom{\ell-j}{i}.$\end{proof}

\begin{proposition}\label{formula} $B(\sigma_1,\dots,\sigma_n)=B(\sigma_1,\dots,\sigma_{n-1})+B(\sigma_n)-2B(\sigma_1\cup \sigma_n,\dots,\sigma_{n-1}\cup \sigma_n)$.\end{proposition}

\begin{proof}  From Proposition \ref{size}, \begin{equation}\label{start} B(\sigma_1,\dots,\sigma_n)=\sum\limits_{i=1}^nB(\sigma_i)
 -2\sum\limits_{i_1,i_2}B(\sigma_{i_1}\cup \sigma_{i_2})
+\dots+(-2)^{n-1}B(\sigma_1\cup\dots\cup \sigma_n).\end{equation}  Rearranging so that all terms involving $\sigma_n$ on the right-hand side of \ref{start} are isolated yields

\begin{align}\label{next}  \left[\sum\limits_{i=1}^{n-1}B(\sigma_i)
 -2\sum\limits_{i_1,i_2\neq n}B(\sigma_{i_1}\cup \sigma_{i_2})
+\dots+(-2)^{n-2}B(\sigma_1\cup\dots\cup \sigma_{n-1})\right]\\ \notag + \left[B(\sigma_n) -2\sum\limits_{i=1}^{n-1}B(\sigma_i\cup \sigma_n) +\cdots +(-2)^{n-1}B(\sigma_1\cup \cdots \cup \sigma_n)\right].\end{align} Again, by Proposition \ref{size}, \begin{equation}\label{n-1} B(\sigma_1,\dots,\sigma_{n-1})=\sum\limits_{i=1}^{n-1}B(\sigma_i)
 -2\sum\limits_{i_1,i_2\neq n}B(\sigma_{i_1}\cup \sigma_{i_2})
+\dots+(-2)^{n-2}B(\sigma_1\cup\dots\cup \sigma_{n-1}),\end{equation} and, since $\sigma_i\cup \sigma_j\cup \sigma_n=(\sigma_i\cup \sigma_n)\cup(\sigma_j\cup \sigma_n)$, \begin{eqnarray}\label{withn} &B(\sigma_1\cup \sigma_n,\dots,\sigma_{n-1}\cup \sigma_n)&\\ \notag &=\sum\limits_{i=1}^{n-1}B(\sigma_i\cup \sigma_n)
 -2\sum\limits_{i_1,i_2}B(\sigma_{i_1}\cup \sigma_{i_2}\cup \sigma_n)
+\dots+(-2)^{n-2}B(\sigma_1\cup\dots\cup \sigma_n).&\end{eqnarray} Substituting \ref{n-1} and \ref{withn} back into the right-hand side of \ref{next} yields the claimed formula.\end{proof}

\begin{proposition}\label{last} Let $\abs{\sigma_1}\leq\abs{\sigma_2}\leq\dots\leq\abs{\sigma_k}\leq\abs{\sigma_{k+1}}=\dots=\abs{\sigma_{n-1}}=\abs{\sigma_n}$ for $k\leq n$.  If $\abs{\sigma_k}<(n-k)+\abs{\sigma_n}$, then $B(\sigma_1,\dots,\sigma_n)\geq B(\sigma_n)$. \end{proposition}
\begin{proof}
The proof will be by induction on $n$.  The base case $n=2$ is covered by the proof of Proposition \ref{base}.  Thus, assume $B(\sigma_1,\dots,\sigma_x)\geq B(\sigma_x)$ for $x<n$.  It must be demonstrated that $B(\sigma_1,\dots,\sigma_n)\geq B(\sigma_n)$.  By the formula in Proposition \ref{formula}, $$B(\sigma_1,\dots,\sigma_n)=B(\sigma_1,\dots,\sigma_{n-1})+B(\sigma_n)-2B(\sigma_1\cup \sigma_n,\dots,\sigma_{n-1}\cup \sigma_n).$$  Note that $B(\sigma_1\cup \sigma_n,\dots,\sigma_{n-1}\cup \sigma_n)$ is the Hamming weight of $n-1$ row vectors being added together, fulfilling the property that all nonzero entries of the rows $\sigma_i\cup\sigma_n$ are also nonzero entries in the row $\sigma_n$.  Hence, $$B(\sigma_n)\geq B(\sigma_1\cup \sigma_n,\dots,\sigma_{n-1}\cup \sigma_n).$$  Thus, $$B(\sigma_1,\dots,\sigma_n)\geq B(\sigma_1,\dots,\sigma_{n-1})+B(\sigma_n)-2B(\sigma_n)$$$$=B(\sigma_1,\dots,\sigma_{n-1})-B(\sigma_n).$$ Then using Proposition \ref{formula} repeatedly, we have $$B(\sigma_1,\dots,\sigma_n)\geq B(\sigma_1,\dots,\sigma_{n-2})+B(\sigma_{n-1})-2B(\sigma_1\cup \sigma_{n-1},\dots,\sigma_{n-2}\cup \sigma_{n-1})-B(\sigma_n)$$$$\geq B(\sigma_1,\dots,\sigma_{n-2})+B(\sigma_{n-1})-2B(\sigma_{n-1})-B(\sigma_n)$$$$=B(\sigma_1,\dots,\sigma_{n-2})-B(\sigma_{n-1})-B(\sigma_n)$$$$\geq \ldots\geq B(\sigma_1,\dots,\sigma_k)-(B(\sigma_{k+1})+\dots+B(\sigma_n)).$$  By the induction hypothesis, $$B(\sigma_1,\dots,\sigma_n)\geq B(\sigma_1,\dots,\sigma_k)-(B(\sigma_{k+1})+\dots+B(\sigma_n))$$$$\geq B(\sigma_k)-(B(\sigma_{k+1})+\dots+B(\sigma_n))$$$$=B(\sigma_k)-(n-k)B(\sigma_n),$$ since $\abs{\sigma_{k+1}}=\dots=\abs{\sigma_n}$.  Let $T_{k+m}\subseteq [\ell ]$ such that $\abs{T_{k+m}}=\abs{\sigma_k}+m$ (so $B(\sigma_n)\leq B(T_{k+m}$) for $m\leq (n-k)$) and recall from Proposition \ref{base} that the Hamming weight between groups of rows decreases by more than half (also, note that $T_{k+i}$ is a row with higher Hamming weight than the row $T_{k+i+1}$).  Thus, $$B(\sigma_1,\dots,\sigma_n)\geq B(\sigma_k)-(n-k)B(\sigma_n)$$$$\geq [B(\sigma_k) - B(T_{k+1})] -B(T_{k+2})-\dots-B(\sigma_n)$$$$\geq [B(T_{k+1})-B(T_{k+2})] - B(T_{k+3}) - \dots - B(\sigma_n)$$$$\geq\ldots\geq B(T_{n-1})-B(\sigma_n)\geq B(\sigma_n).$$  Note that there are exactly enough $T_{k+m}$ terms for the above inequalities since $\abs{\sigma_k}<(n-k)+\abs{\sigma_n}$.  Thus, for $\abs{\sigma_k}<(n-k)+\abs{\sigma_n}$, $B(\sigma_1,\dots,\sigma_n)\geq B(\sigma_n)$.  \end{proof}

With respect to Equation \ref{hard} Propositions \ref{size}, \ref{base}, \ref{formula},  and \ref{last} give evidence for the following conjecture.

\begin{conjecture}\label{conjecture}
 
The minimum distance of $K(\ell ,h,j)$ is $$d=\sum\limits_{i=1}^{h-j}\binom{\ell-j}{i}.$$
 
\end{conjecture}

\section{Hamming codes}\label{Hamming}

Let $\mathcal{H}_\ell$ be the binary $\ell^{\mathrm{th}}$ Hamming code. Then the following proposition shows that the Boolean subspace arrangement codes where the simplicial complex is homotopically a sphere with $j=\ell -2$ is a Hamming code.
\begin{proposition}
 The codes $K(\ell,\ell-1,\ell-2)$ and $\mathcal{H}_\ell$ are permutation equivalent.
\end{proposition}
\begin{proof}
By Theorem 1.8.1 in Huffman and Pless \cite{WCFVP}, any $[2^r-1,2^r-1-r,3]$ code is equivalent to $\mathcal{H}_r$.  The length of $K(\ell, \ell-1, \ell-2)$ by Equation \ref{k-length} is $$\sum\limits_{i=0}^h\binom{\ell}{i}=\sum\limits_{i=0}^{\ell-1}\binom{\ell}{i}=\sum\limits_{i=0}^{\ell}\binom{\ell}{i}-\binom{\ell}{\ell}=2^\ell-1.$$  The dimension of $K(\ell,\ell-1,\ell-2)$ by Corollary \ref{k-dimension} is $$\sum\limits_{i=0}^j f_i=\sum\limits_{i=0}^{\ell-2} f_i=\sum\limits_{i=0}^{\ell-2}\binom{\ell}{i}=\sum\limits_{i=0}^\ell\binom{\ell}{i}-\binom{\ell}{\ell-1}-\binom{\ell}{\ell}=2^\ell-\ell-1=2^\ell-1-\ell.$$  The upper bound on the minimum distance Lemma \ref{1-row} is by $$d\leq \sum\limits_{i=0}^{h-j}\binom{\ell-j}{i}=\sum\limits_{i=0}^{\ell-(\ell-1)}\binom{\ell-(\ell-2)}{i}=\sum\limits_{i=0}^1\binom{2}{i}=\binom{2}{0}+\binom{2}{1}=3.$$  By Corollary 1.4.14 in \cite{WCFVP}, since the rows in the matrix generating $K(\ell,\ell-1,\ell-2)$ are linearly independent, $d\geq 3$, implying $d=3$.  Thus, $K(\ell,\ell-1,\ell-2)\cong\mathcal{H}_\ell.$\end{proof}

\begin{remark}
 Since the two codes are permutation-equivalent, there must be some means by which to transform $K(\ell,\ell-1,\ell-2)$ into $\mathcal{H}_\ell$.  This process is achieved by adding rows, permuting columns, and permuting rows in the matrix generating $K(\ell,\ell-1,\ell-2)$.  Since the matrix (if the points in $V(I)$ are ordered appropriately) can be made into an upper-triangular matrix adjoined to a $(2^\ell-\ell-1)\times\ell$ matrix, adding all rows to rows above them will result in a $[I_{2^\ell-\ell-1}|A]$ matrix.  In the $A$ block (corresponding to points in $V(I)$ that have only one 0 [and thus $\ell-1$ ones] as a component), there are (since there are $\ell-2$ rows) $\sum\limits_{0}^{\ell-2-j}\binom{\ell-1-j}{i}=2^{\ell-1-j}-1$ ones, which is odd.  Thus, when all the rows are added, the $A$ block remains unchanged. This $A$ block consists of binary strings of length $\ell$ that corresponding to polynomials up to degree $\ell-2$ being evaluated on points with $\ell-1$ ones.  These strings, once transposed and combined with a transposed identity matrix, will consist of all numbers in binary from 1 to $2^\ell$, indicating its equivalence to $\mathcal{H}_\ell$.
\end{remark}

At this time the authors do not know, but suspect that there are other values of $\ell$, $h$, and $j$ where the codes $K(\ell ,h,j)$ are permutaion equivalent to other well known codes.

\bibliographystyle{amsplain}

\bibliography{bergbib}

\end{document}